%% file: schur.tex
 \pdfoutput=1

\documentclass[
    paper       = A4,
    parskip     = half,
    draft,
    fontsize    = 11pt,
    paper       = A4,
    DIV         = classic,
    headsepline = true
]{scrartcl}

\usepackage{tikz}

\usepackage[math]{cellspace}
\newcommand{\name}[1]{\textsc{#1}}

\newcounter{enumintertext}
\makeatletter
\newcommand{\enumintertext}[1]{%
\setcounter{enumintertext}{\value{enum\romannumeral\the\@enumdepth}}
\end{enumerate}
#1
\begin{enumerate}
\setcounter{enum\romannumeral\the\@enumdepth}{\value{enumintertext}}%
}
\makeatother

\input{math}
\input{notations}

\input{language}

\newtheorem{theorem}{Theorem}[section]

\newtheorem{lemma}[theorem]{Lemma}
\newtheorem{proposition}[theorem]{Proposition}

\newcommand\MoveEqRight[1][2]{\kern -#1em  &   \kern #1em}

\usepackage[numbers]{natbib}

\usepackage[
    final         = true,
    debug         = true
]{hyperref}

\hypersetup{
    pdftitle            = {Non-symmetric polarization},
    pdfauthor           = {Andreas Defant and Sunke Schlüters},
    pdfkeywords         = {Schur multipliers},
    pdfproducer         = {LaTeX /w hyperref}
}

\title{Non-symmetric polarization}
\author{Andreas Defant and Sunke Schlüters}

\begin{document}

    \maketitle

\begin{abstract}
    \noindent   Let $P$ be an $m$-homogeneous polynomial in $n$-complex variables $x_1, \dotsc, x_n$. Clearly, $P$ has a unique representation in the form
    \begin{equation*}
        P(x)= \sum_{1 \leq j_1 \leq \dotsc \leq j_m \leq n}
                c_{(j_1, \dotsc, j_m)}
                \,
                x_{j_1} \dotsb x_{j_m}
        \,,
    \end{equation*}
    and the $m$"~form 
    \begin{equation*}
        L_P(x^{(1)}, \dotsc, x^{(m)})= \sum_{1 \leq j_1 \leq \dotsc \leq j_m \leq n}
                c_{(j_1, \dotsc, j_m)}
                \, 
                x^{(1)}_{j_1} \dotsb x^{(m)}_{j_m}
    \end{equation*}
    satisfies $L_P(x,\dotsc, x) = P(x)$ for every $x\in\C^n$. We show that, although $L_P$ in general is non-symmetric, for a large class of reasonable norms $ \norm[]{ \phdot } $ on $\C^n$
    the  norm of $L_P$  on $(\C^n, \norm[]{ \phdot } )^m$ up
    to a logarithmic term $(\const \log n)^{m^2}$ can be estimated by the norm of $P$ on $ (\C^n, \norm[]{ \phdot } )$; here $\const \ge 1$ denotes a universal constant. Moreover, for the $\ell_p$"~norms $ \norm[]{ \phdot }_p$, $1 \leq p < 2$ the logarithmic term in the number $n$ of variables is even superfluous.
\end{abstract}


\section{Introduction}

It is well-known that for every $m$"~homogeneous polynomial $P : \C^n \to \C$ there is a unique symmetric $m$"~linear form $L : (\C^n)^m \to \C$ such that $L(x,\dotsc, x) = P(x)$ for all $x \in \C^n$. Uniqueness is an immediate consequence of the well-known \emph{polarization formula} (see e.g. \cite[Section~1.1]{D99}): For each $m$"~homogeneous polynomial $P : \C^n \to \C$ and each symmetric $m$"~form $L$ on $\C^n$ such that $P(x) = L(x,\dotsc, x)$ for every $x \in \C^n$, we have for every choice of $x^{(1)}, \dotsc, x^{(m)}\in \C^n$
\begin{equation*}
    L\big( x^{(1)}, \dotsc, x^{(m)} \big)
    = \frac{1}{2^m m!} \sum_{\epsilon_k = \pm 1}
        \epsilon_1 \dotsb \epsilon_m P\Big(
            \sum_{k=1}^m \epsilon_k x^{(k)}
        \Big)
        \, .
\end{equation*}
Moreover, as an easy consequence, for each norm $\norm{ \phdot }$ on $\C^n$
\begin{equation}
    \label{eq:polarization_formula_norm_upper}
    \sup_{\norm{x^{(k)}} \le 1}
        \abs[\big]{ L \big( x^{(1)}, \dotsc, x^{(m)} \big) }
    \le \e^m \cdot 
        \sup_{\norm{x}\le1}
        \abs[]{ P(x) }
        \, .
\end{equation}

Existence can be seen as follows: Every $m$"~homogeneous polynomial $P : \C^n \to \C$ has a unique representation of the form
\begin{equation*}
    P(x)= \sum_{1 \leq j_1 \leq \dotsc \leq j_m \leq n}
            c_{(j_1, \dotsc, j_m)}
            \,
            x_{j_1} \dotsb x_{j_m}
    \, .
\end{equation*}
A $m$"~form on $\C^n$ which is naturally associated to $P$ is given by
\begin{equation*}
    L_P(x^{(1)}, \dotsc, x^{(m)}) \ceq
        \sum_{1 \leq j_1 \leq \dotsc \leq j_m \leq n}
            c_{(j_1, \dotsc, j_m)}
            \, 
            x^{(1)}_{j_1} \dotsb x^{(m)}_{j_m}
        \, ,
\end{equation*}
and the symmetrization $\Sym L_P$, defined by
\begin{equation*}
    \Sym L_P \big( x^{(1)}, \dotsc, x^{(m)} \big)
        \ceq \frac{1}{m!} \sum_\sigma
            L_P \big(x^{(\sigma(1))}, \dotsc, x^{(\sigma(1))} \big)
        \, ,
\end{equation*}
where the sum runs over all $\sigma \in \Sigma_m$ (the set of all permutations of the first $m$ natural numbers), then is the unique symmetric $m$"~form satisfying $L(x,\dotsc, x) = P(x)$ for every $x \in \C^n$.

\bigskip

Note that $L_P$ is in general not symmetric. For an arbitrary non-symmetric multilinear form $L : (\C^n)^m \to \C$ and the associated polynomial $P(x) \ceq L(x,\dotsc, x)$ we have in general no estimate as in \eqref{eq:polarization_formula_norm_upper}. Take for example $L:(\C^n)^2 \to \C$ defined by $(x,y) \mapsto x_1 y_2 - x_2 y_1$. Then $P(x) = L(x,x) = 0$, but $L \not= 0$.

Our purpose is now to establish estimates as in \eqref{eq:polarization_formula_norm_upper} for the multilinear form $L_P$ instead of $\Sym L_P$. The norms $\norm{ \phdot }$ we consider on $\C^n$ are $1$"~unconditional, i.e. $x,y \in \C^n$ with $ \abs[]{ x_k } \le \abs[]{ y_k }$ for every $k$ implies $\norm{x} \le \norm[]{ y }$. Examples are the $\ell_p$"~norms $\norm{\phdot}_p$ for $1\le p \le \infty$.

 Our main result is the following:

\begin{theorem}
    \label{thm:mt_poly}
    There exists a universal constant $\const_1 \ge 1$ such that for every $m$"~homogeneous polynomial $P : \C^n \to \C$ and every $1$"~unconditional norm  $\norm{ \phdot }$ on $\C^n$
    \begin{equation}
        \label{eq:mt_1}
        \sup_{\norm{x^{(k)}}\le1}
            \abs[\big]{ L_P \big( x^{(1)}, \dotsc, x^{(m)} \big) }
        \le (\const_1 \log n)^{m^2} \cdot 
            \sup_{\norm{x}\le1}
            \abs[]{ P(x) }
            \, .
    \end{equation}

    Moreover, if $\norm{ \phdot } = \norm[]{ \phdot }_p$ for $1\le p < 2$, then there even is a constant $\const_2 = \const_2(p) \ge 1$ for which
    \begin{equation}
        \label{eq:mt_2}
        \sup_{\norm{x^{(k)}}\le1}
            \abs[\big]{ L_P \big( x^{(1)}, \dotsc, x^{(m)} \big) }
        \le \const_2^{m^2} \cdot 
            \sup_{\norm{x}\le1}
            \abs[]{ P(x) }
            \, .
    \end{equation}
\end{theorem}

Bearing \eqref{eq:polarization_formula_norm_upper} in mind, it suffices to establish the inequality
\begin{equation*}
    \sup_{\norm{x^{(k)}}\le1}
            \abs[\big]{ L_P \big( x^{(1)}, \dotsc, x^{(m)} \big) }
    \le \const \cdot
        \sup_{\norm{x^{(k)}}\le1}
            \abs[\big]{ \Sym L_P \big( x^{(1)}, \dotsc, x^{(m)} \big) }
\end{equation*}
with a suitable constant $\const$. We will prove this inequality by iteration, based on the following theorem. For $1\le k \le n$ define the partial symmetrization $\Sym_k L_p : (\C^n)^m \to \C$ of $L_P$ by
\begin{equation*}
    \Sym_k L_P \big( x^{(1)}, \dotsc, x^{(m)} \big)
        \ceq \frac{1}{k!} \sum_{\sigma \in \Sigma_k}
                L_P\big(
                    x^{(\sigma(1))}, \dotsc, x^{(\sigma(k))},
                    x^{(k+1)}, \dotsc, x^{(m)}
                 \big)
    \, .
\end{equation*}

\begin{theorem}
    \label{thm:mt_mlform}
    There exists a universal constant $\const_1 \ge 1$ such that for every $m$"~homogeneous polynomial $P : \C^n \to \C$, every $1$"~unconditional norm $\norm{\phdot}$ on $\C^n$ and $1 \le k \le m$
    \begin{equation*}
        \sup_{\norm{x^{(k)}}\le1}
                \abs[\big]{ \Sym_{k-1} L_P \big( x^{(1)}, \dotsc, x^{(m)} \big) }
        \le (\const_1 \log n)^k \cdot
            \sup_{\norm{x^{(k)}}\le1}
                \abs[\big]{ \Sym_k L_P \big( x^{(1)}, \dotsc, x^{(m)} \big) }
            \, .
    \end{equation*}

    Moreover, if $\norm{ \phdot } = \norm[]{ \phdot }_p$ for $1\le p < 2$, then there even is a constant $\const_2 = \const_2(p) \ge 1$ for which
    \begin{equation*}
        \sup_{\norm{x^{(k)}}\le1}
                \abs[\big]{ \Sym_{k-1} L_P \big( x^{(1)}, \dotsc, x^{(m)} \big) }
        \le \const_2^k \cdot
            \sup_{\norm{x^{(k)}}\le1}
                \abs[\big]{ \Sym_k L_P \big( x^{(1)}, \dotsc, x^{(m)} \big) }
            \, .
    \end{equation*}
\end{theorem}

The proofs require the theory of \name{Schur} multipliers, which was initiated by \citet{Schur_1911}. As a crucial tool we will use norm estimates for the \emph{main triangle projection} due to \citet{KP70} as well as \citet{B77} (see also \cite{ST12,ST2013} and \cite{Defant_2011}).


\section{Comparing coefficients}

A $m$"~linear form $L : (\C^n)^m \to \C$ is uniquely determined by its coefficients
\begin{equation*}
    c_\bfi(L) \ceq L(e_{i_1} , \dotsc, e_{i_m}) \, , \quad \bfi \in \I(n,m) \ceq \set[]{ 1,\dotsc,n  }^m
    \, ,
\end{equation*}
where $e_k$ denotes the $k$\textsuperscript{th} canonical basis vector in $\C^n$. With $L_\bfi : (\C^n)^m \to \C$ defined by $(x^{(1)}, \dotsc, x^{(m)}) \mapsto x^{(1)}_{i_1} \dotsm x^{(m)}_{i_m}$ we see at once that
\begin{equation}
    \label{eq:multilinear_coefficients}
    L = \sum_{\bfi \in \I(n,m)} c_\bfi(L) \, L_\bfi
    \, .
\end{equation}

The index set $\I(n,m)$ carries a natural equivalence relation: $\bfi, \bfj \in\I(n,m)$ are equivalent, notation $\bfi \sim \bfj$, if there exists a permutation $\sigma \in \Sigma_m$ of the first $m$ natural numbers such that $i_k = j_{\sigma(k)}$ for every $k$. The equivalence class of $\bfi \in \I(n,m)$ will be denoted by $[\bfi]$. It is easy to check that for every $\bfi \in \I(n,m)$ there exists a unique $\bfj \in \J(n,m) \ceq \set[]{ (j_1, \dotsc, j_m) \in \I(n,m) \given j_1 \le j_2 \le \dotsb \le j_m }$ such that $[\bfi] = [\bfj]$, respectively $\bfi \sim \bfj$. We will use the symbol $\bfi^\ast$ to denote this unique index $\bfj$. For $\bfi \in \I(n,m_1)$ and $\bfj \in \I(n,m_2)$ we write $(\bfi,\bfj) \in \I(n,m_1+m_2)$ for the concatenation of the two.

The main idea of the proofs is now to compare $c_\bfi(\Sym_k L_P)$ and $c_\bfi(\Sym_{k-1}L_P)$. For this let us compute $c_\bfi(\Sym_k L_P)$.

\begin{lemma}
    \label{lem:coeff_of_Sk}
    Let $P: \C^n \to \C$ be an $m$"~homogeneous polynomial and $\bfi \in \I(n,m)$. Then
    \begin{equation*}
        c_\bfi ( \Sym_k L_P )
        = \frac{c_{\bfi^\ast}(L_P)}{ \abs[]{ [(i_1,\dotsc,i_k)] } }
    \end{equation*}
    if $(i_{k+1}, \dotsc, i_m) \in \J(n,m-k)$ and $\max \{ i_1, \dotsc, i_k \} \le i_{k+1}$; and otherwise
    \begin{equation*}
        c_\bfi ( \Sym_k L_P ) = 0
        \, .
    \end{equation*}
\end{lemma}

\begin{proof}
    By definition we have 
    \begin{align*}
        c_\bfi( \Sym_k L_P )
        &= \Sym_k L_P( e_{i_1}, \dotsc, e_{i_m} ) \\
        &= \frac{1}{k!} \sum_{\sigma\in\Sigma_k}
            \sum_{\bfj \in \J(n,m)} c_\bfj(L_P)
            \, L_\bfj(
                e_{i_{\sigma(1)}}, \dotsc, e_{i_{\sigma(k)}},
                e_{i_{k+1}}, \dotsc, e_{i_m}
            ) \\
        &= \sum_{\bfj \in \J(n,m)} c_\bfj(L_P)
            \, 
            \frac{1}{k!} \sum_{\sigma\in\Sigma_k}
            L_\bfj(
                e_{i_{\sigma(1)}}, \dotsc, e_{i_{\sigma(k)}},
                e_{i_{k+1}}, \dotsc, e_{i_m}
            )
        \, .
    \end{align*}
    Now, $L_\bfj(
            e_{i_{\sigma(1)}}, \dotsc, e_{i_{\sigma(k)}},
            e_{i_{k+1}}, \dotsc, e_{i_m}
        )$ equals $1$ if $\bfj = ( i_{\sigma(1)}, \dotsc, i_{\sigma(k)}, i_{k+1}, \dotsc, i_m
        )$ and vanishes otherwise. Thus
    \begin{equation*}
        c_\bfi( \Sym_k L_P )
        = \frac{ c_{\bfi^\ast} (L_P) }{ k! } 
            \cdot \abs[\big]{ \set[]{ \sigma\in\Sigma_k \given ( i_{\sigma(1)}, \dotsc, i_{\sigma(k)}, i_{k+1}, \dotsc, i_m
            ) \in \J(n,m)  } } 
            \, .
    \end{equation*}

    If $(i_{k+1}, \dotsc, i_m) \not\in \J(n,m-k)$ or $\max \{ i_1, \dotsc, i_k \} > i_{k+1}$, then there doesn't exist any permutation $\sigma \in \Sigma_k$ for which $( i_{\sigma(1)}, \dotsc, i_{\sigma(k)}, i_{k+1}, \dotsc, i_m ) \in \J(n,m)$. If not, then there are
    \begin{equation*}
        \frac{k!}{ \abs[]{ [(i_1,\dotsc, i_k)] } }
    \end{equation*}
    many permutations $\sigma \in \Sigma_k$ for which $i_{\sigma(1)} \le i_{\sigma(2)} \le \dotsb \le i_{\sigma(k)}$.
\end{proof}

\begin{proposition}
    \label{prop:coeff_compare}
    Let $P: \C^n \to \C$ be an $m$"~homogeneous polynomial, $\bfi \in \I(n,m)$ and $k\in \set[]{ 2, \dotsc, m }$. Then
    \begin{equation*}
        c_\bfi(\Sym_{k - 1} L_P) = \frac{k}{ \abs[\big]{
                \set[]{ 1 \le u \le k \given i_u = i_k }
            } } \cdot c_\bfi( \Sym_k L_P )
    \end{equation*}
    provided $\max \set[]{ i_1, \dotsc, i_{k-1} } \le i_k$; and otherwise
    \begin{equation*}
        c_\bfi(\Sym_{k - 1} L_P) = 0 \cdot c_\bfi( \Sym_k L_P )
        \, .
    \end{equation*}
\end{proposition}

For the proof we need an additional lemma.
\begin{lemma}
    \label{lem:card_is}
    For every $\bfi \in \I(n, k)$
    \begin{equation*}
        \abs[]{ [\bfi] }
        = \abs[]{ [(i_1,\dotsc,i_{k-1})] }
            \cdot \frac{k}{ \abs[]{ \set[]{ 1\le  u \le k \given i_u = i_k } } }
            \, .
    \end{equation*}
\end{lemma}

\begin{proof}
    Let us first examine the quantity $\abs{ [\bfi] }$ for $\bfi \in \I(n,k)$. An easy combinatorial argument shows that
    \begin{equation*}
        \abs[]{ [\bfi] } = \frac{k!}{\alpha_1! \alpha_2! \dotsm \alpha_n!}
        \, ,
    \end{equation*} 
    where $\alpha_l \ceq \abs[]{ \set[]{ 1 \le u \le k \given i_u = l } }$, $1\le l \le n$; note that the numerator counts all permutations of the first $k$ natural numbers and the denominator counts those permutations which give the same index.

    Let now $\beta_l \ceq \abs[]{ \set[]{ 1 \le u \le k -1 \given i_u = l } }$. Then $\alpha_l = \beta_l + 1$ for $l = i_k$ and $\alpha_l =\beta_l$ for all $l \not= i_k$. Thus
    \begin{equation*}
        \abs[]{ [\bfi] } 
        = \frac{k!}{\alpha_1! \alpha_2! \dotsm \alpha_n!}
        = \frac{(k-1)!}{\beta_1! \dotsm \beta_n!} \cdot \frac{k}{\alpha_{i_k}}
        = \abs[]{ [(i_1,\dotsc,i_{k-1})] }
            \cdot \frac{k}{ \abs[]{ \set[]{ 1\le u\le k \given i_u = i_k } } }
            \, .
            \qedhere
    \end{equation*}
\end{proof}

\begin{proof}[Proof of Proposition~\ref{prop:coeff_compare}]
    Let $\bfk \in \I(n,m)$. We decompose $\bfk = (\bfi, l, \bfj) \in \I(n,m)$ with $\bfi \in \I(n,k-1)$, $l \in \set[]{ 1, \dotsc, n } = \I(n,1)$, and $\bfj \in \I(n,m-k)$. Using Lemma~\ref{lem:coeff_of_Sk}, the following table distinguishes three cases for the $\bfk$\textsuperscript{th} coefficient of $\Sym_kL_P$ and $\Sym_{k-1} L_P$:
    \begin{center}
        \begin{tabular}{Sc Sc >{\hspace{1em}}Sc >{\hspace{2em}}Sc}
            &
            & $c_\bfk(\Sym_k L_P)$
            & $c_\bfk(\Sym_{k-1} L_P)$ \\
        \hline
            \emph{(1)}
            & \parbox{10em}{
                \centering
                $\bfj \in \J(n,m-k)$ \\
                $l \le j_1$ \\
                $\max \set[]{ i_1, \dotsc, i_{k-1} } \le l$
            }
            & $\displaystyle \frac{1}{ \abs[]{ [(\bfi,l)] } }
                                \cdot c_{\bfk^\ast}(L_P)$
            & $\displaystyle \frac{1}{ \abs[]{ [\bfi] } }
                                \cdot c_{\bfk^\ast}(L_P)$            \\
        \hline
            \emph{(2)}
            & \parbox{12em}{
                \centering
                $\bfj \in \J(n,m-k)$ \\
                $l \le j_1$ \\
                $l < \max \set[]{ i_1, \dotsc, i_{k-1} }  \le j_1$ 
            }
            & $\displaystyle \frac{1}{ \abs[]{ [(\bfi,l)] } }
                                \cdot c_{\bfk^\ast}(L_P)$
            & $0$                                       \\
        \hline
            \emph{(3)}
            & otherwise
            & $0$
            & $0$                                       \\
        \end{tabular}
    \end{center}

    In case \emph{(1)} we deduce by Lemma~\ref{lem:card_is}, as desired
    \begin{align*}
        c_\bfk(\Sym_{k-1} L_P)
        &= \frac{ c_{\bfk^\ast}(L_P) }{ \abs[]{ [\bfi] } } 
        = \frac{ \abs[]{ [(\bfi,l)] } }{ \abs[]{ [\bfi] } }
            \cdot \frac{ c_{\bfk^\ast}(L_P) }{ \abs[]{ [(\bfi,l)] } } 
        = \frac{ \abs[]{ [(\bfi,l)] } }{ \abs[]{ [\bfi] } }
            \cdot c_\bfk(\Sym_k L_P) \\
        \MoveEqRight[10]
        = \frac{ k }{ \abs[]{ \set[]{ 1 \le u \le k \given i_u = l }} }
            \cdot c_\bfk(\Sym_k L_P)
        \, ,
    \end{align*}
    and in the cases \emph{(2)} and \emph{(3)} the conclusion is evident.
\end{proof}


\section{Multidimensional and classical \name{Schur} multipliers}

Let $c_\bfi(A)$ denote the $\bfi$\textsuperscript{th} entry of a matrix $A \in \C^{\I(n,m)}$. For $A, B \in \C^{\I(n,m)}$ the ($m$"~dimensional) \name{Schur} product $A \schur B \in \C^{\I(n,m)}$ is defined by
\begin{equation*}
    c_\bfi( A \schur B ) \ceq c_\bfi(A) \cdot c_\bfi(B) \,.
\end{equation*}
Having \eqref{eq:multilinear_coefficients} in mind, the \name{Schur} product of a $m$"~form $L : (\C^n)^m \to \C$ and $A \in \C^{\I(n,m)}$ is given by
\begin{equation*}
    A \schur L \ceq \sum_{\bfi} \big( c_\bfi(A) \cdot c_\bfi(L) \big) \, L_\bfi
    \, .
\end{equation*}

Recall that by Proposition~\ref{prop:coeff_compare} for each $1\le k \le m$ we have $\Sym_{k - 1} L_P = \mathfrak A^k \schur \Sym_k L_P$, where $\mathfrak A^k \in \C^{\I(n,m)}$ is defined by
\begin{equation*}
    c_\bfi( \mathfrak A^k ) \ceq \frac{k}{ \abs[\big]{
            \set[]{ 1 \le u \le k \given i_u = i_k }
        } }
\end{equation*}
if $\max \set[]{ i_1, \dotsc, i_{k-1} } \le i_k$; and $c_\bfi(\mathfrak A^k ) \ceq 0$ otherwise. Let us decompose $\mathfrak A^k$ into the \name{Schur} product of more handily pieces. For $u,v \in \set[]{ 1,\dotsc, m  }$ let $D^{u,v}\in \C^{\I(n,m)}$ be defined by $c_\bfi(D^{u,v}) \ceq 1$ if $i_u = i_v$ and $c_\bfi(D^{u,v}) \ceq 0$ otherwise. Define furthermore $T^{u,v} \in \C^{\I(n,m)}$ by $c_\bfi(T^{u,v}) \ceq 1$ if $i_u \le i_v$ and $c_\bfi(T^{u,v}) \ceq 0$ if $i_u > i_v$.

With these definitions $\mathfrak A^k$ decomposes as follows.
\begin{lemma}
    \label{lem:Ak_ingredients}
    For $1 \le k \le m$ we have
    \begin{equation}
        \label{eq:matrixA}
        \mathfrak{A}^k =
            \Big( \schur_{u=1}^{k-1} T^{u,k} \Big)
            \schur \bigg(
                \sum_{u=1}^{k} \frac{k}{u} \cdot
                A^{k,u}
            \bigg)
    \end{equation}
    with
    \begin{equation*}
        A^{k,u} \ceq
                    \sum_{\substack{
                        Q \subset \set[]{ 1, \dotsc, k } \\
                        \abs[]{ Q } = u
                    }}
                    \Big( \schur_{q \in Q} D^{q,k} \Big)
                    \schur \Big( \schur_{q \in Q^c} (\mathbf{1} - D^{q,k}) \Big)
                \, ,
    \end{equation*}
    where $Q^c$ denotes the complement of $Q$ in $\set[]{ 1, \dotsc,  k }$ and $\mathbf{1} \in \C^{\I(n,m)}$ is defined by $c_\bfi(\mathbf 1)=1$ for all $\bfi$.
\end{lemma}

\begin{proof}
    Throughout the proof, we will denote the right-hand side of \eqref{eq:matrixA} by $A^k$. Let $\bfi \in \I(n,m)$. If there exists some $1 \le u \le k-1$ such that $i_u > i_k$, then we have by definition $c_\bfi(\mathfrak A^k) = 0$. On the other hand, in this case $c_\bfi(T^{u,k}) =0$ and thus $c_\bfi(A^k) = 0$.

    Assume now that $i_u \le i_k$ for all $1\le u \le k$. Then $c_\bfi(T^{u,k}) = 1$ for all $1\le u \le k-1$. With $Q_\bfi \ceq \set[]{ 1\le u \le k \given i_u = i_k }$ we check at once that
    \begin{equation*}
        c_\bfi\bigg(
            \Big( \schur_{q \in Q} D^{q,k} \Big)
                    \schur \Big( \schur_{q \in Q^c} (\mathbf{1} - D^{q,k}) \Big)
        \bigg) = \begin{cases}
            1 & \text{if $Q = Q_\bfi$,} \\
            0 & \text{if $Q \not= Q_\bfi$.}
        \end{cases}
    \end{equation*}
    Therefore $c_\bfi(A^{k,u})$ evaluates to $1$ if $u = \abs[]{ Q_\bfi }$ and vanishes otherwise. We have
    \begin{equation*}
        c_\bfi(A^k) = \frac{k}{\abs[]{ Q_\bfi }}
            = \frac{k}{ \abs[\big]{
                \set[]{ 1 \le u \le k \given i_u = i_k }
            } }
            = c_\bfi(\mathfrak A^k) \, .
            \qedhere
    \end{equation*} 
\end{proof}

\bigskip
We have seen that $D^{u,v}$ and $T^{u,v}$ are the building blocks of $\mathfrak A^k$ under \name{Schur} multiplication. In what follows we will investigate the \name{Schur} norms of these matrices.

For a given norm $ \norm[]{ \phdot } $ on $\C^n$ and $A\in\C^{\I(n,m)}$ we denote by $\mu^m_{\scriptscriptstyle \norm[]{ \phdot }}(A)$ the best constant $\const$ such that
\begin{equation*}
    \sup_{\norm[]{ x^{(k)} } \le 1}
            \abs[\big]{ A \schur L \big(x^{(1)}, \dotsc, x^{(m)} \big) }
        \le \const \cdot \sup_{\norm[]{ x^{(k)} } \le 1}
            \abs[\big]{ L \big(x^{(1)}, \dotsc, x^{(m)} \big) }
\end{equation*}
for any $m$"~form $L : (\C^n)^m \to \C$.

\begin{lemma}
    \label{lem:D^uv_T^uv}
    For every $n,m$, every $u,v \in \set[]{ 1,\dotsc, m}$, and every $1$"~unconditional norm $\norm{ \phdot} $ on $\C^n$
    \begin{gather}
        \label{eq:D^uv}
        \mu^m_{\scriptscriptstyle \norm[]{ \phdot }}(D^{u,v}) = 1
        \, , \\
        \label{eq:T^uv_log}
        \mu^m_{\scriptscriptstyle \norm[]{ \phdot }}( T^{u,v} ) \le \log_2(2n)
        \, .
    \end{gather}
    Moreover, for every $1 \le p < 2$ there exists a constant $\const_3= \const_3(p)$ so that for every $n,m$ and $u,v \in \set[]{ 1,\dotsc, m}$
    \begin{equation}
        \label{eq:T^uv_p}
        \mu^m_{\scriptscriptstyle \norm[]{ \phdot }_p}( T^{u,v} ) \le \const_3
        \, .
    \end{equation}
\end{lemma}

To prove this lemma we have to resort to the classical theory of \name{Schur} multipliers. Define $T_n = (t^n_{ij})_{i,j} \in \C^{n \times n}$ by
\begin{equation*}
     t^n_{ij} = \begin{cases}
        1 & i\le j \le n, \\
        0 & \text{otherwise,}
    \end{cases}
\end{equation*}
and let $I_n \in \C^{n \times n}$ denote the identity matrix.

\begin{lemma}
    \label{lem:I_n,T_n}
    We have for every $n$
    \begin{gather}
        \mu_{\scriptscriptstyle \norm[]{ \phdot }_\infty}^2(I_n)   \le 1 \, , \label{eq:I_n} \\
        \mu_{\scriptscriptstyle \norm[]{ \phdot }_\infty}^2(T_n) \le \log_2(2n) \, , \label{eq:T_n_log}
    \intertext{and, moreover, for $1 \le p < 2$ there is a constant $\const_3 = \const_3(p)$ such that for every $n$}
        \mu_{\scriptscriptstyle \norm[]{ \phdot }_p}^2(T_n)      \le \const_3 \, . \label{eq:T_n_p}
    \end{gather}
\end{lemma}

These inequalities are due to \citet{KP70} as well as \citet{B76}. More precisely, Proposition~1.1 of \cite{KP70} gives for any matrix $(a_{ij})_{i,j} \in \C^{n\times n}$
\begin{equation*}
    \sup_{\substack{ \norm[]{ x }_\infty \le1 \\ \norm[]{ y }_\infty \le 1}}
        \abs[\Big]{
            \sum_{i,j = 1}^n t_{ij}^n a_{ij} y_i x_j
        }
    \le \log_2 ( 2n ) \cdot \sup_{\substack{ \norm[]{ x }_\infty \le1 \\ \norm[]{ y }_\infty \le 1}}
        \abs[\Big]{
            \sum_{i,j = 1}^n a_{ij} y_i x_j
        }
        \, ,
\end{equation*}
which is \eqref{eq:T_n_log}. Statement \eqref{eq:T_n_p} follows from Theorem~5.1 of \cite{B76}, which (implicitly) states that for $1 \le p < 2$
\begin{equation*}
        \sup_{\substack{ \norm[]{ x }_p \le1 \\ \norm[]{ y }_p \le 1}}
        \abs[\Big]{
            \sum_{i,j = 1}^n t_{ij}^n a_{ij} y_i x_j
        }
        \le \const_3(p) \cdot
            \sup_{\substack{ \norm[]{ x }_p \le1 \\ \norm[]{ y }_p \le 1}}
            \abs[\Big]{
                \sum_{i,j = 1}^n a_{ij} y_i x_j
            }
            \, .
\end{equation*}

For the proof of \eqref{eq:I_n} recall that by Theorem~4.3 of \cite{B77} and the duality $\ell_\infty^n = (\ell_1^n)'$ we have that
\begin{equation}
    \label{eq:thm4.3}
    \mu^2_{\scriptscriptstyle \norm[]{ \phdot }_\infty}(I_n)
    = \sup_{\substack{ d \in \C^n \\ \norm[]{ d }_\infty \le1}}
        \pi_1 \big( \ell_1^n \overset{\operatorname{diag} d}\longrightarrow
                    \ell_1^n \overset{I_n}\longrightarrow
                    \ell_\infty^n
            \big)
        \, ,
\end{equation}
where the $1$"~summing norm $\pi_1$ of an operator $T : X \to Y$ in finite dimensional spaces is defined as (see e.g. \cite{DJT95} or \cite{Defant_1993})
\begin{equation*}
    \pi_1(T) \ceq \sup \set[\bigg]{ 
        \sum_{k=1}^l \norm[]{ Tx_k }_Y
    \given
        l \in \N, x_k \in X, \sup_{\abs[]{ t_k } = 1} \norm[\Big]{ \sum_{k=1}^l t_k x_k }_X \le 1
    }
    \, .
\end{equation*}
By the ideal property of $\pi_1$ and the well-known fact that $\pi_1(\ell_1^n \overset{I_n}\longrightarrow \ell_\infty^n) = 1$ (see \cite[Section~22.4.12]{Pietsch_1980} or \cite[Section~10.4 and 11.1]{Defant_1993}) the right-hand side of \eqref{eq:thm4.3} equals $1$.

\begin{proof}[Proof of Lemma~\ref{lem:D^uv_T^uv}]
    We begin with the proof of \eqref{eq:D^uv} for the supremum norm $ \norm[]{ \phdot }_\infty$ on $\C^n$. Let $L : (\C^n)^m \to \C$ be a multilinear form. Without loss of generality we may assume $u=1$ and $v=2$. Then
    \begin{align*}
        \MoveEqLeft[4]
        \sup_{\norm[]{ x^{(k)} }_\infty \le 1}
            \abs[\big]{ D^{u,v} \schur L \big(x^{(1)}, \dotsc, x^{(m)} \big) }\\ 
        &= \sup_{\norm[]{ x^{(k)} }_\infty \le 1}
            \abs[\Big]{ \sum_{\bfi \in \I(n,m)} d_{i_1 i_2} c_\bfi(L) \, x^{(1)}_{i_1} \dotsm x^{(m)}_{i_m} }
            \\
        &= \sup_{\substack{
            x^{(3)}, \dotsc, x^{(m)} \\
            \norm[]{ x^{(k)} }_\infty \le 1
        }}
            \sup_{\substack{
            x^{(1)}, x^{(2)} \\
            \norm[]{ x^{(k)} }_\infty \le 1
        }}
            \abs[\Big]{
                \sum_{i,j=1}^n d_{i j} 
                    \bigg( \sum_{\substack{\bfi \in \I(n,m) \\ i_1 = i \\ i_2 = j}} c_\bfi(L) \, x^{(3)}_{i_3} \dotsm x^{(m)}_{i_m} \bigg)
                    x^{(1)}_i x^{(2)}_j
            } \, .
    \intertext{Using \eqref{eq:I_n}, we see that this is}
        &\le \sup_{\norm[]{ x^{(k)} }_\infty \le 1}
            \abs[\Big]{ \sum_{\bfi \in \I(n,m)} c_\bfi(L) \, x^{(1)}_{i_1} \dotsm x^{(m)}_{i_m} } \\
        &= \sup_{\norm[]{ x^{(k)} }_\infty \le 1}
            \abs[\big]{ L\big( x^{(1)},\dotsc, x^{(m)}) } 
            \, ,
    \end{align*}
    which proves $\mu^m_{\scriptscriptstyle \norm[]{ \phdot }_\infty}(D^{u,v}) = 1$. In a second step we now show that this inequality holds for any given $1$"~unconditional norm $ \norm[]{ \phdot }$ on $\C^n$. Again, let $L : (\C^n)^m \to \C$ be an $m$"~form and fix $x^{(1)}, \dotsc, x^{(m)} \in \C^n$ so that $\norm{ x^{(k)} } \le 1$. With $\tilde L : (\C^n)^m \to \C$ defined by
    \begin{equation*}
        \tilde L \big( y^{(1)}, \dotsc, y^{(m)} \big)
        \ceq L ( x^{(1)} \cdot y^{(1)}, \dotsc, x^{(m)} \cdot y^{(m)} )
        \, ,
    \end{equation*}
    where $x^{(k)} \cdot y^{(k)} \ceq (x^{(k)}_1 \cdot y^{(k)}_1, \dotsc, x^{(k)}_n \cdot y^{(k)}_n)$, we deduce from the first part of this proof that
    \begin{align*}
        \abs[\big]{ D^{u,v} \schur L \big( x^{(1)}, \dotsc, x^{(m)} \big) }
        & \le \sup_{\substack{
            y^{(1)}, \dotsc, y^{(m)} \\
            \norm[]{ y^{(k)} }_\infty \le 1
        }} \abs[\big]{ D^{u,v} \schur \tilde L \big( y^{(1)}, \dotsc, y^{(m)} \big) }
            \\
        \MoveEqLeft[4]
        \le \sup_{\substack{
            y^{(1)}, \dotsc, y^{(m)} \\
            \norm[]{ y^{(k)} }_\infty \le 1
        }} \abs[\big]{ \tilde L \big( y^{(1)}, \dotsc, y^{(m)} \big) }
        \le \sup_{\substack{
            y^{(1)}, \dotsc, y^{(m)} \\
            \norm[]{ y^{(k)} } \le 1
        }} \abs[\big]{ L \big( y^{(1)}, \dotsc, y^{(m)} \big) }
        \, ;
    \end{align*}
    note that the last inequality holds true due to the $1$"~unconditionality of $\norm{ \phdot }$.

    The proof of \eqref{eq:T^uv_log} follows the same lines using \eqref{eq:T_n_log} instead of \eqref{eq:I_n}. Finally, to prove \eqref{eq:T^uv_p} one only has to use the first step of the preceding argument with the norm $\norm{ \phdot }_\infty$ replaced by $\norm{ \phdot }_p$ and \eqref{eq:I_n} substituted by \eqref{eq:T_n_p}.
\end{proof}


\section{Proof of the Theorems~\ref{thm:mt_poly} and \ref{thm:mt_mlform}}

We are now ready to give the proofs of the Theorems~\ref{thm:mt_poly} and \ref{thm:mt_mlform}. We begin with  Theorem~\ref{thm:mt_mlform}, as Theorem~\ref{thm:mt_poly} will then follow easily.

\begin{proof}[Proof of Theorem~\ref{thm:mt_mlform}]
    Note at first that for any $1$"~unconditional norm $\norm{ \phdot }$ on $\C^n$ the \name{Schur} norm $\mu^m_{\scriptscriptstyle \norm[]{ \phdot }}$ turns the linear space $\C^{\I(n,m)}$ into an \name{Banach} algebra. By Lemma~\ref{lem:Ak_ingredients} and \eqref{eq:D^uv},
    \begin{align*}
        \mu^m_{\scriptscriptstyle \norm{\phdot}} (A^{u,k})
        \le \sum_{\substack{
                        Q \subset \set[]{ 1, \dotsc, k } \\
                        \abs[]{ Q } = u
                    }}
                \Big( \prod_{q\in Q}  \mu^m_{\scriptscriptstyle \norm{\phdot}} (D^{q,k}) \Big)
                \cdot
                \Big( \prod_{q\in Q^c} \underbrace{ \mu^m_{\scriptscriptstyle \norm{\phdot}} (\mathbf{1} - D^{q,k}) }_{\le \, 2} \Big)
                &
            \\
        \MoveEqLeft[10]
        \le \sum_{\substack{
                        Q \subset \set[]{ 1, \dotsc, k } \\
                        \abs[]{ Q } = u
                    }} 1^{\abs[]{ Q }} 2^{\abs[]{ Q^c }}
        = \binom{k}{u} 2^{k-u}
        \, ,
    \end{align*}
    and thus
    \begin{align*}
        \mu^m_{\scriptscriptstyle \norm{\phdot}}(\mathfrak{A}^k)
        &\le \mu^m_{\scriptscriptstyle \norm{\phdot}} \Big( \schur_{u=1}^{k-1} T^{u,k} \Big)
            \cdot \mu^m_{\scriptscriptstyle \norm{\phdot}} \bigg(
                \sum_{u=1}^{k} \frac{k}{u} \cdot
                A^{k,u}
            \bigg)
                \\
        \MoveEqRight[4]
        \le \big(
                \mu^m_{\scriptscriptstyle \norm{\phdot}} ( T^{u,k} )
            \big)^{k-1} \cdot k
                \sum_{u=1}^{k} \binom{k}{u} 2^{k-u}
        \le k 3^k \big(
                \mu^m_{\scriptscriptstyle \norm{\phdot}} ( T^{u,k} )
            \big)^{k-1}
        \, .
    \end{align*}
    Finally, the results in \eqref{eq:T^uv_log} and \eqref{eq:T^uv_p} complete the proof.
\end{proof}

We remark that the best constants $\const_1$ and $\const_2$ in Theorem~\ref{thm:mt_mlform} satisfy the estimates $(\const_1 \log n)^k \le k 3^k \big( \log_2 (2n) \big)^{k-1}$ and $\const_2^k \le k 3^k \const_3^{k-1}$ with $\const_3$ denoting the constant in \eqref{eq:T^uv_p}.

We finish with the proof of our main theorem.

\begin{proof}[Proof of Theorem~\ref{thm:mt_poly}]
    Repeated application of Theorem~\ref{thm:mt_mlform} yields 
    \begin{align*}
        \MoveEqLeft[4]
        \sup_{\norm{x^{(k)}}\le1}
            \abs[\big]{ L_P \big( x^{(1)}, \dotsc, x^{(m)} \big) }
        \\
        &= \sup_{\norm{x^{(k)}}\le1}
            \abs[\big]{ \Sym_1 L_P \big( x^{(1)}, \dotsc, x^{(m)} \big) }
        \\
        &\le \const^2 \cdot \sup_{\norm{x^{(k)}}\le1}
            \abs[\big]{ \Sym_2 L_P \big( x^{(1)}, \dotsc, x^{(m)} \big) }
        \\
        &\le \dotsb
        \\
        &\le \const^{2 + \dotsb + (m-1) + m}
            \cdot \sup_{\norm{x^{(k)}}\le1}
            \abs[\big]{ \Sym_m L_P \big( x^{(1)}, \dotsc, x^{(m)} \big) }
            \, ,
    \end{align*}
    with $\const$ denoting the respective constants of Theorem~\ref{thm:mt_mlform}. Finally, \eqref{eq:polarization_formula_norm_upper} (which is an immediate consequence of the polarization formula) completes the argument (note that by definition $\Sym L_P = \Sym_m L_P$).
\end{proof}


\end{document}

%% file: math.tex
\usepackage{amsmath, amsthm, amsfonts, amssymb}
\usepackage{mathtools}

%% file: notations.tex
\usepackage{amsopn}

\newcommand{\N}{\mathbb{N}}     
\newcommand{\C}{\mathbb{C}}     

\newcommand{\I}{\mathcal{I}}
\newcommand{\J}{\mathcal{J}}

\newcommand{\bfi}{{\boldsymbol{i}}}
\newcommand{\bfj}{{\boldsymbol{j}}}
\newcommand{\bfk}{{\boldsymbol{k}}}

\DeclareFontFamily{OT1}{pzc}{}
\DeclareFontShape{OT1}{pzc}{m}{it}{<-> s * [1.10] pzcmi7t}{}
\DeclareMathAlphabet{\mathpzc}{OT1}{pzc}{m}{it}

\newcommand{\e}{\mathrm{e}}
\renewcommand{\epsilon}{\varepsilon}
\newcommand{\const}{c}

\DeclarePairedDelimiter{\abs}{\lvert}{\rvert}
\DeclarePairedDelimiter{\norm}{\lVert}{\rVert}

\newcommand{\ceq}{\vcentcolon=}

\newcommand{\Sym}{\mathcal{S}}

\newcommand{\phdot}{\,\cdot\,}				

\DeclareMathOperator*{\schur}{\mathchoice
	{\vcenter{\hbox{\large\ensuremath\ast}}}
	{\vcenter{\hbox{\large\ensuremath\ast}}}
	{\vcenter{\hbox{\footnotesize\ensuremath\ast}}}
	{\vcenter{\hbox{\scriptsize\ensuremath\ast}}}
}

\DeclarePairedDelimiterX{\innerproduct}[2]{(}{)}{#1 \,\delimsize\vert\, #2}
\DeclarePairedDelimiterX{\dualpairing}[2]{\langle}{\rangle}{#1 , #2}

\DeclarePairedDelimiterX\set[1]{\{}{\}}{%
    \providecommand\given{\,\delimsize\vert\,}
    #1
}


%% file: language.tex
\usepackage[utf8x]{inputenc}
\usepackage[ngerman,english]{babel}

\usepackage{ucs}

\makeatletter
    \initiate@active@char{"}
    \addto\extrasenglish{\languageshorthands{english}\bbl@activate{"}}
    \addto\noextrasenglish{\bbl@deactivate{"}}
    \declare@shorthand{english}{"~}{\nobreak\hbox{--}\bbl@allowhyphens}
\makeatother
